\documentclass[a4paper,11pt]{article}
\usepackage[utf8]{inputenc}
\usepackage[T1]{fontenc}
\usepackage[english]{babel}
\usepackage{amsmath,amssymb,amsthm}
\usepackage[short]{optidef}				
\usepackage[shortlabels]{enumitem} 
\usepackage{comment}
\usepackage{standalone}
\usepackage{mathtools}
\usepackage{fullpage}
\usepackage{float}
\usepackage{mathdots}
\usepackage{tikz}
\usepackage[subrefformat=parens,labelformat=parens]{subcaption}
\graphicspath{{./Figuras/}} 
\usetikzlibrary{graphs,graphs.standard}
\usepackage{hyperref}
\usepackage{thmtools}
\usepackage{thm-restate}
\tikzstyle{line} = [draw, thick]

\definecolor{myblue}{RGB}{80,80,160}
\definecolor{mygreen}{RGB}{80,160,80}
\definecolor{myred}{RGB}{255,0,0}
\definecolor{mybrown}{RGB}{139,69,19}

\DeclarePairedDelimiter\ceil{\lceil}{\rceil}
\DeclarePairedDelimiter\floor{\lfloor}{\rfloor}

\newtheorem{theorem}             {Theorem}
        
\newtheorem{conjecture}	[theorem] {Conjecture}

\newtheorem{proposition}[theorem] {Proposition}
\newtheorem{remark}[theorem] {Remark}
\newtheorem{observation}[theorem] {Observation}
\newtheorem{problem}[theorem] {Problem}   
\newtheorem{corollary}	[theorem] {Corollary}
     
\newtheorem{claim}{Claim}

\newtheoremstyle{case}{}{}{}{}{\bfseries}{:}{ }{}
\theoremstyle{case}
\newtheorem{case}{Case}

\numberwithin{subcase}{case}

\newcommand{\authormark}[1]{\textsuperscript{\,#1}}
\newcommand{\showmark}[1]{%
  \hspace*{-1em}\makebox[1em][r]{\authormark{#1}\,}%
  \ignorespaces}

\begin{document}

\title{Towards a Dual Version of Woodall’s Conjecture for Partial 3-Trees}

\author{Juan Gutiérrez} 
\date{}

\maketitle

\begin{center}
\footnotesize

\bigskip
\showmark{1}
Departamento de Ciencia de la Computación\\ 
Universidad de Ingeniería y Tecnología (UTEC), Lima, Perú\\
E-mail: \texttt{jgutierreza@utec.edu.pe}

\end{center}

\begin{abstract}
	A dual version of a conjecture by Woodall asserts that, in a planar
digraph, the length of a shortest dicycle equals the maximum number of pairwise disjoint feedback arc sets. We verify this conjecture for the case where the underlying graph is a 3-tree or a partial 3-tree with girth $3$.
Additionally, we show that every 3-tree has a feedback arc set of size at most~$m/3-1$, where~$m$ is the number of arcs of the digraph, and this bound is tight. We further establish an upper bound on the size of a minimum feedback arc set in $k$-trees. Finally, we discuss some open problems and conjectures.
\end{abstract}

\section{Introduction}	\label{sec:introduction}

A \textit{graph} is a pair~$(V,E)$, where~$V$ is a set
of \textit{vertices} and~$E$ is a set of unordered pairs of distinct vertices, called \textit{edges}.
A \textit{directed graph}, or \textit{digraph},
is a pair~$(V,E)$, where~$V$ is a set
of \textit{nodes} and~$E$ is a set of ordered pairs of distinct
nodes, called \textit{arcs}.
A \emph{dicut} in a digraph~$G$ is a partition of~$V(G)$ into two subsets, so that each arc with ends in both subsets leaves the same set of the partition,
and a \textit{dijoin} is a subset of~$E(G)$ that intersects every dicut.

In 1978, Woodall conjectured that the size of a maximum number of pairwise disjoint dijoins is equal to the size of a minimum dicut \cite{Woodall78}.
Feofiloff and Younger, and Schrijver, independently proved that this
holds for source-sink connected digraphs \cite{Feofiloff87,Schrijver1982},
and Lee and Wakabayashi established it when the underlying graph is
series-parallel \cite{Lee2001}.
Recently, Cornuéjols et al.   \cite{Cornuejols2025} showed it when the underlying graph is chordal.

In this paper, we are interested in a dual version of Woodall's conjecture.
A directed cycle, or \textit{dicycle}, in a digraph is a directed walk in which all vertices are distinct except that the first and last vertices coincide.
A directed path, or \textit{dipath}, in a digraph is a directed walk in which all vertices are distinct.
Given a digraph, a
\textit{feedback arc set}
is a set of arcs that intersects every dicycle of the digraph. A \textit{packing} is a set of pairwise disjoint feedback arc sets.
We denote by~$\nu(G)$ the size of a maximum packing in a digraph~$G$ and by~$g(G)$ the \emph{girth} of $G$, that is, the length of a shortest dicycle in~$G$.
As every dicycle must intersect any feedback arc set, we have~$\nu(G) \leq g(G)$.

It is straightforward to observe that for planar digraphs, every dicut corresponds to a dicycle,
and every dijoin corresponds to a feedback arc set in the dual digraph.
Hence, by duality, Conjecture~\ref{conj:Woodall-planar} is equivalent to asserting that,
for every planar digraph, its girth
equals the cardinality of a maximum packing.

\begin{conjecture}[{\cite{Woodall78,Lee2001}}]\label{conj:Woodall-planar}
  For every planar digraph~$G$ with at least one dicycle,~$\nu(G)=g(G)$.
\end{conjecture}

This dual version of the conjecture has received less attention than the original one.
Lee and Wakabayashi \cite{Lee2001} showed that Conjecture~\ref{conj:Woodall-planar} is true for digraphs whose underlying graph is series-parallel
and Lee and Williams showed that it
is true when the underlying graph is planar and has no~$K_5-e$ minor \cite{Lee2006}.
We are not aware of any other result regarding Conjecture~\ref{conj:Woodall-planar}.

In this paper, we prove the next result.

\begin{restatable}{theorem}{maintheorem}
\label{th:pl3tree}
If~$G$ is a digraph, with at least one dicycle, whose underlying graph is a 3-tree, then~$\nu(G)= g(G)$.
\end{restatable}

Note that Theorem~\ref{th:pl3tree} is valid even if the digraph is not planar. However, the hypothesis of planarity of Conjecture~\ref{conj:Woodall-planar} cannot be removed in the general case.
In fact, Donadelli and Kohayakawa \cite{Donadelli2002} exhibited a tournament~$G$ on 15 vertices\footnote{They attribute this construction to Thomassen, via a private communication between Younger and Wakabayashi.} for which~$\beta(G) > m(G)/3$, where $\beta(G)$ is the size of a minimum feedback arc set in $G$. Hence, since the girth of any tournament with at least one dicycle equals 3, $\nu(G) \neq g(G)$.
Moreover, they showed an infinite family of digraphs showing that the planarity hypothesis cannot be dropped from Conjecture~\ref{conj:Woodall-planar}. 

As in the previous paragraph, an easy corollary to observe is that if Conjecture~\ref{conj:Woodall-planar}  
holds, then~$\beta(G) \leq m(G)/g(G)$.
Hence, it is interesting to study a weaker version of Conjecture~\ref{conj:Woodall-planar} by removing the hypothesis of planarity, since it can yield upper bounds on the size of a minimum feedback arc set.

In this direction, very recently, Gutin et al. \cite{Gutin2025}
obtained several results considering the maximum degree~$\Delta(G)$ of a graph.
For instance, they 
showed that 
$\nu(G) \geq k$ if~$\Delta(G) \leq 3$ and~$g(G) \geq k$ for~$k\in \{3,4,5\}$.
This implies that Conjecture~\ref{conj:Woodall-planar}
holds for digraphs whose underlying graph is subcubic and whose directed girth is between 3 and 5.
It is interesting to observe that this property holds even when~$G$ is not planar.
They also showed that~$\nu(G) < g(G)$ for some digraphs with~$g(G) \in \{4,6,10\}$.
Moreover, an open conjecture they gave is whether~$\nu(G)=3$ 
when~$g(G)\geq 3$ and~$\Delta(G)=5$.

It is straightforward to note that if~$G$ is a digraph with at least one dicycle whose underlying graph is a 3-tree then~$g(G)=3$ (Proposition~\ref{prop:dycicleimpliesditriangle}). Hence, a direct corollary of Theorem~\ref{th:pl3tree} is that~$\beta(G) \leq m(G)/3$ when~$G$ is a digraph whose underlying graph is a 3-tree. Moreover, we can easily extend this result for partial 3-trees.

\begin{corollary}\label{cor:betam3partial3trees}
Let~$G$ be a digraph whose underlying graph is a partial 3-tree.
If~$g(G) \geq 3$
    then~$\beta(G) \leq m(G)/3$.
    Moreover, if $g(G)=3$
    then~$\nu(G)=g(G)$.
\end{corollary}
\begin{proof}

We direct any pair of vertices not in~$E(G)$ 
arbitrarily, creating a digraph~$G'$ whose underlying graph is a 3-tree.
As $G$ is not acyclic, $G'$ is not acyclic.
As we have not created antiparallel arcs in $G'$, we have $3\leq g(G')$.
So $g(G')=3$ by Proposition~\ref{prop:dycicleimpliesditriangle}.
By Theorem~\ref{th:pl3tree},~$\nu(G')=g(G')$.
As any dicycle in~$G$ exists also in~$G'$, any packing of~$G'$ is also a packing of~$G$, so we have~$3=g(G')=\nu(G') \leq \nu(G)$.
Hence, as $\nu(G) \geq 3$, $G$ has a feedback arc set of size at most $m(G)/3$.
Now, if $g(G)=3$, then $3 \leq \nu(G) \leq g(G)=3$,
so $\nu(G) =g(G)$.
\end{proof}

A natural question is whether the upper bound in the first part of Corollary~\ref{cor:betam3partial3trees} is asymptotically tight for partial 3-trees. We show that, in fact, this is the case for 3-trees.

\begin{restatable}{theorem}{threetreetightTheorem}\label{thm:3treetight}
   Let~$G$ be a digraph whose underlying graph is a 3-tree.
    If~$g(G) \geq 3$ then~$\beta(G) \leq \frac{m(G)-3}{3}$.
    Moreover, this bound is tight.
\end{restatable}

Furthermore, we also prove the next result for~$k$-trees.

\begin{restatable}{theorem}{ktreetightTheorem}\label{thm:betaktree}
Let~$G$ be a digraph whose underlying graph is a $k$-tree.
    If~$g(G) \geq 3$ then~~$\beta(G) \leq \frac{\floor {k/2}}{k}m(G) + \frac{k-1}{4}(\ceil{k/2}-\floor {k/2}) - c(k+1)^{3/2}$, where~$c= \frac{1}{8 \sqrt{\pi}}$.
\end{restatable}

\section{Preliminaries}	

For a given graph $G$, we denote by $N_G(u)$ the set of neighbors of $u$ in $G$. A clique is a set of pairwise adjacent vertices,
and a triangle is a clique of size $3$.
In this paper, we are interested in studying Conjecture~\ref{conj:Woodall-planar} when the underlying graph has bounded treewidth.
The concept of treewidth is highly relevant in the proof of the graph minor theorem by Robertson and Seymour \cite{Robertson90}
as well as in the field of combinatorial optimization \cite{Bodlaender88,Arnborg89}.


A graph~$G$ is a~$k$-tree if and only if $n(G) \geq k+1$ and
there exists an ordering 
$v_1,v_2, \ldots, v_n$ 
of the vertices of~$G$ such that
$N_{G_i}(v_i)$ is a clique of size~$\min\{k,i-1\}$,
where~$G_i$ is the graph induced by~$v_1,v_2,\ldots,v_i$.
We call such an ordering a
\emph{$k$-elimination ordering} for $G$.
A \emph{partial~$k$-tree} is any graph that is a subgraph of a~$k$-tree.
It is known that partial~$k$-trees are exactly the graphs with treewidth at most~$k$ \cite{Bodlaender98}.

Partial 1-trees are precisely the forests. Partial 2-trees are the series-parallel graphs \cite{Brandstadt99}.
Lee and Wakabayashi proved that Conjecture~\ref{conj:Woodall-planar} holds for digraphs whose underlying graph is a partial 2-tree.
Hence, it seems natural to answer Conjecture~\ref{conj:Woodall-planar} for
planar digraphs whose underlying graph is a partial 3-tree.

\begin{problem}\label{prob:woodalltw3}
Show Conjecture~\ref{conj:Woodall-planar} for
 digraphs whose underlying graph is a partial 3-tree.
\end{problem}

Given two graphs~$H$ and~$G$, we say that 
$H$ is a \textit{minor} of~$G$ if~$H$ can be
obtained from~$G$ by deleting 
edges, vertices and by contracting edges. 
Lee and Williams showed that Conjecture \ref{conj:Woodall-planar} holds
when the underlying graph has no~$K_5-e$ minor \cite{Lee2001}.
The next observation shows that~$K_5-e$ minor-free graphs are partial 3-trees,
and thus the result given by Lee and Williams partially solves Problem~\ref{prob:woodalltw3}.
We will use the characterization of partial 3-trees given by Arnborg \cite{Arnborg1990}.


\begin{theorem}[{\cite[Theorem 1.2]{Arnborg1990}}]\label{th:arnborg}
The set of minimal forbidden minors for partial 3-trees
consists of four graphs:~$K_5, M_6, M_8$, and~$M_{10}$ (Figure~\ref{fig:forbidden-minors}).
\end{theorem}

	\begin{figure}[H]
		\centering
		\includegraphics[scale=1]{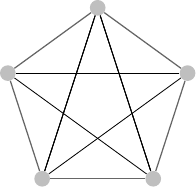}
		\includegraphics[scale=1]{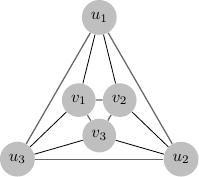}	
			\includegraphics[scale=1]{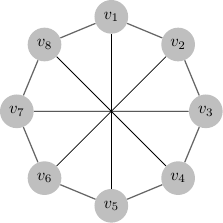}
		\includegraphics[scale=1]{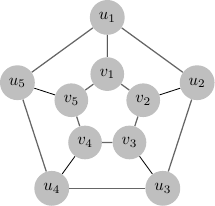}
		\caption{The minimal forbidden minors for partial 3-trees. From left to right:~$K_5, M_6, M_8$, and~$M_{10}$.}
		\label{fig:forbidden-minors}		
	\end{figure}

\begin{observation}\label{prop:K5-ehastw3}
Every~$K_5-e$ minor-free graph is a partial 3-tree.
\end{observation}
\begin{proof}

Let~$G$ be a~$K_5-e$ minor-free graph.
Suppose by contradiction that~$G$ is not a partial 3-tree.
By Theorem~\ref{th:arnborg},~$G$ has one of
$K_5, M_6, M_8$ or~$M_{10}$ as a minor.
If~$G$ has a~$K_5$ minor, then it is clear that it also has a~$K_5-e$ minor, and we obtain a contradiction.
Suppose now that~$G$ has~$M_6$ as a minor, and suppose
the vertices of~$M_6$ are 
as in Figure~\ref{fig:forbidden-minors}.
Then, by contracting edge~$u_1v_1$ we obtain a~$K_5-e$, a contradiction.
Suppose now that~$G$ has~$M_8$ as a minor, and suppose
the vertices of~$M_8$ are
as in Figure~\ref{fig:forbidden-minors}.
Then, by contracting edges~$v_1v_2, v_3v_4$ and~$v_6v_7$ we obtain a~$K_5-e$, a contradiction.
Finally, let us suppose that~$G$ has~$M_{10}$ as a minor, and suppose
the vertices of~$M_{10}$ are
as in Figure~\ref{fig:forbidden-minors}.
Then, by contracting edges~$v_1v_2, v_4v_5,u_1u_5,u_2u_3$ and~$u_3u_4$ we obtain a~$K_5-e$, again a contradiction.
\end{proof}


In this paper, we progress towards Problem~\ref{prob:woodalltw3},
tackling the case
when the underlying graph is a 3-tree.
Note that planar 3-trees are not included in the class
of~$K_5-e$ minor-free graphs.
Indeed, the~$K_5-e$ is itself a planar 3-tree, and by adding new vertices, we can obtain an infinite family of planar 3-trees with~$K_5-e$ as a minor.
Moreover, we can remove the hypothesis of planarity,
solving Problem~\ref{prob:woodalltw3} for 3-trees.


\begin{restatable}{theorem}{maintheorem}
\label{th:3tree}
If~$G$ is a digraph with at least one dicycle and whose underlying graph is a 3-tree, then~$\nu(G)=g(G)$.
\end{restatable}

We finish this section with the next known propositions (see \cite{Bodlaender98}).
We let~$V_3$ be the degree~3 vertices of a 3-tree~$G$.
A \emph{separator} in a connected graph is a subset of vertices whose removal disconnects the graph.

 \begin{proposition} \label{prop:G-V3}
  If~$G$ has at least 5 vertices, then~$V_3$ is an independent set, and~$G-V_3$ is a 3-tree or a triangle. 
 \end{proposition}

 \begin{proposition} \label{prop:seppluscomps}
 If~$G$ is a~$3$-tree with a separator~$S$ of size~$3$, then~$C \cup S$ induces a~$3$-tree for any component~$C$ of~$G-S$.
 \end{proposition}  

  \begin{proposition} \label{prop:trianglesinG-V3areseparator}
If~$G$ is a 3-tree on at least 5 vertices, then every triangle in~$G-V_3$ is a separator in~$G$.
 \end{proposition}

A \emph{ditriangle} in a digraph is a dicycle of size 3. 
A graph is called \emph{chordal} if every induced cycle
has size 3. It is known that $k$-trees are chordal \cite{Bodlaender98}.
\begin{proposition}\label{prop:dycicleimpliesditriangle}
Let~$G$ be a digraph whose underlying graph is a 3-tree. If
$G$ is not acyclic and
$g(G) \geq 3$, then
$g(G) = 3$.
\end{proposition}

For the rest of the paper, we set $n(G):=|V(G)|, m(G):=|E(G)|$ for any digraph $G$, and we use $m$ and $n$ if $G$ is clear from the context.
 
\section{Proof Theorem \ref{th:3tree}} \label{sec:pl3trees}





For the case when~$g(G)=2$, see Remark~\ref{remark:dec2acyclic}. Hence, by Proposition~\ref{prop:dycicleimpliesditriangle}, we may assume that~$g(G)=3$.
We prove, by induction on~$n$, that~$G$ has a packing of size 3.
If~$n \leq 4$, then~$G$ is one of the digraphs in Figure~\ref{fig:tournaments4vertices}, and the claim follows by inspection.
Suppose now that~$n>4$. We divide the rest of the proof in two cases.

\label{fig:tournaments4vertices}
\begin{figure}[h]
\centering
\begin{tikzpicture}[scale=1.3, every node/.style={circle, fill=black, inner sep=1.5pt}, ->, >=stealth]

\def\spacing{3.0}  
\def\topY{0}
\def\triangleA{(0,1)}
\def\triangleB{(-1,-0.5)}
\def\triangleC{(1,-0.5)}
\def\center{(0,0)}

\begin{scope}[shift={(0*\spacing, \topY)}]
    \node (a) at \triangleA {};
    \node (b) at \triangleB {};
    \node (c) at \triangleC {};
    \node (d) at \center    {};

    \draw (a) -- (b);
    \draw (a) -- (c);
    \draw (a) -- (d);
    \draw (b) -- (c);
    \draw (b) -- (d);
    \draw (c) -- (d);
\end{scope}

\begin{scope}[shift={(1*\spacing, \topY)}]
    \node (a) at \triangleA {};
    \node (b) at \triangleB {};
    \node (c) at \triangleC {};
    \node (d) at \center    {};

    \draw (a) -- (b);
    \draw (a) -- (c);
    \draw (a) -- (d);
    \draw (b) -- (c);
    \draw (c) -- (d);
    \draw (d) -- (b);
\end{scope}

\begin{scope}[shift={(2*\spacing, \topY)}]
    \node (a) at \triangleA {};
    \node (b) at \triangleB {};
    \node (c) at \triangleC {};
    \node (d) at \center    {};

    \draw (a) -- (b);
    \draw (b) -- (c);
    \draw (c) -- (d);
    \draw (d) -- (a);
    \draw (a) -- (c);
    \draw (b) -- (d);
\end{scope}

\begin{scope}[shift={(3*\spacing, \topY)}]
    \node (a) at \triangleA {};
    \node (b) at \triangleB {};
    \node (c) at \triangleC {};
    \node (d) at \center    {};

    \draw (a) -- (b);
    \draw (a) -- (c);
    \draw (a) -- (d);
    \draw (b) -- (c);
    \draw (c) -- (d);
    \draw (d) -- (b);
\end{scope}

\end{tikzpicture}

\caption{All non-isomorphic 3-trees with 4 vertices.}
\end{figure}

 \setcounter{case}{0}

 \begin{case}~There exists a ditriangle in~$G$ whose vertices form a separator in the underlying graph, say~$abc$.

 Let~$C_1,C_2, \ldots ,C_p$ be the components of~$G-abc$. For any~$i$,
 let~$G_i = C_i \cup abc$.
 By Proposition~\ref{prop:seppluscomps}, every~$G_i$ is a 3-tree.
 Hence, by induction hypothesis, every~$G_i$ has a packing~$\{T^i_1,T^i_2,T^i_3\}$.
Without loss of generality, we may assume that
$ab \in T^i_1$, $bc \in T^i_2$, and $ca \in T^i_3$ for every such~$i$. We will show that~$\{T_1, T_2,T_3\}=\{\bigcup_i T^i_1, \bigcup_i T^i_2, \bigcup_i T^i_3\}$ is a packing of~$G$.
Let~$C$ be a dicycle in~$G$. 
If~$V(C) \subseteq V(G_i)$ for some~$i$, then~$C \cap T_j \neq \emptyset$ for any~$j \in \{1,2,3\}$ and we are done.
So, from now on, let us assume that $V(C) \not  \subseteq V(G_i)$ for any $i$.
This implies that~${|V(C) \cap \{a,b,c\}| \geq 2}$.

\begin{claim}\label{claim:internallydijsointpaths}
Let~$P$ be a dipath whose ends~$x$ and~$y$ lie in~$abc$.
If~$P$ is internally disjoint from~$abc$, then
$P \cap T_j \neq \emptyset$ for any~$T_j$
that intersects the subpath of~$abc$ from~$x$ to~$y$.
\end{claim}
\begin{proof}
    Let~$Q$ be the subpath of~$abc$ from~$x$ to~$y$.
    Let~$\overline{Q}$ be the subpath of~$abc$ from~$y$ to~$x$.
    Note that~$P \cdot \overline{Q}$ is a dicycle in some~$G_i$.
    So by induction hypothesis, 
   ~$(P \cdot \overline{Q})\cap T_k \neq \emptyset$ for any~$k \in \{1,2,3\}$.
    Let~$j\in \{1,2,3\}$ such that~$T_j \cap Q \neq \emptyset$.
   Since~$T_j \cap \overline{Q} = \emptyset$ and~$(P \cdot \overline{Q})\cap T_j \neq \emptyset$, it follows that~$P \cap T_j \neq \emptyset$.
\end{proof}

Let~$C= P_1 \cdot P_2 \cdots P_{|C|}$, where any~$P_i$ is internally disjoint from~$abc$.
Let~$v_1, v_2, \ldots, v_{|C|},v_{|C|+1}=v_1$
be an ordering of vertices of~$V(C) \cap V(abc)$ such that~$P_i$ has ends~$v_i$ and~$v_{i+1}$.
For any~$i$, let~$Q_i$ be the subpath of~$abc$ from~$v_{i}$ to~$v_{i+1}$. 
By Claim~\ref{claim:internallydijsointpaths},~$P_i \cap T_j \neq \emptyset$ 
if~$Q_i \cap T_j \neq \emptyset$.
As, for any~$j \in \{1,2,3\}$, there exists an~$i$ such that~$Q_i \cap T_j \neq \emptyset$, there exists also an~$i$ such that~$P_i \cap T_j \neq \emptyset$ and the proof follows.
   \end{case}
  \begin{case}~~There exists no ditriangle in~$G$ whose vertices form a separator in the underlying graph.

  Let~$V_3$ be the set of vertices of degree 3 in~$G$, and let~$G'=G-V_3$. In this case, we set~$T_1=E(G')$
and define~$T_2$ and~$T_3$ according to the next rule.
Let~$u \in V_3$ and suppose that the neighbors of~$u$ in the underlying 3-tree are~$a,b$ and~$c$.
Note that~$\{a,b,c\}$ does not induce a ditriangle in~$G$.
Indeed, otherwise Proposition~\ref{prop:trianglesinG-V3areseparator} would imply that~$abc$ is a separator in~$G$.
Thus, we may assume that~$ab,bc,ac \in E(G)$.
If~$xyu$ is a ditriangle, with~$x,y \in \{a,b,c\}$, then we add~$yu$ to~$T_2$ and~$ux$ to~$T_3$.
(Figure~\ref{fig:case2}).

 	\begin{figure}[H]
		\centering
		\includegraphics[scale=0.65]{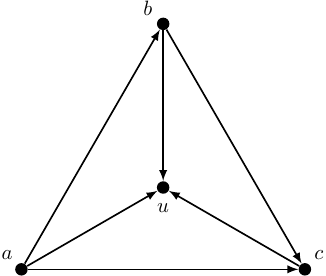}
		\includegraphics[scale=0.65]{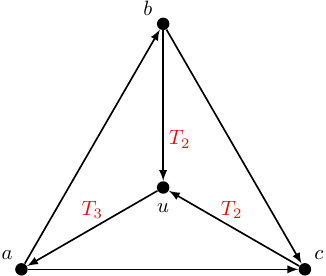}
		\includegraphics[scale=0.65]{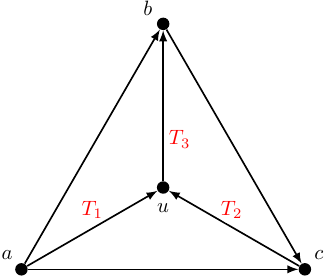}
		\includegraphics[scale=0.65]{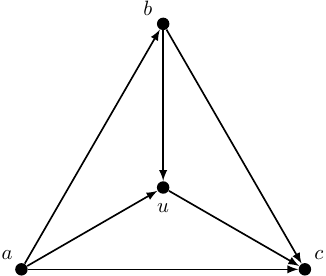}
		\includegraphics[scale=0.65]{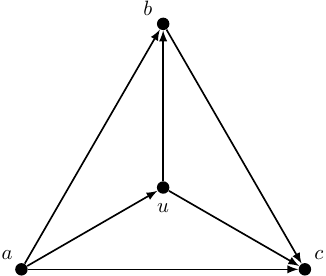}
		\includegraphics[scale=0.65]{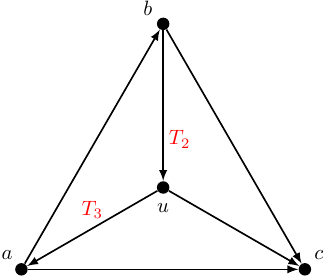}
		\includegraphics[scale=0.65]{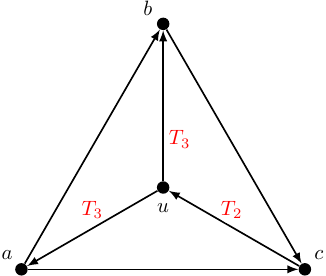}
		\includegraphics[scale=0.65]{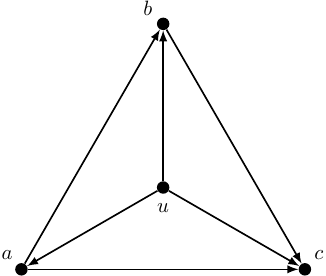}
		\caption{The construction of the packing for the proof of Case 2 in Theorem~\ref{th:pl3tree}. In each of the cases, we label in each arc to which
		feedback arc set this arc will be added. No labelled arcs are not included in anyu feedback arc set of the packing.}
		\label{fig:case2}		
	\end{figure}

We will show that~$\{T_1,T_2,T_3\}$ is a packing of~$G$.
Note that, by Proposition~\ref{prop:G-V3}, the underlying graph of~$G'$
is a 3-tree or a triangle.
Suppose for a moment that~$G'$ has a dicycle. Then, as~$g(G') \geq g(G) \geq 3$, it follows that~$g(G')=3$ by Proposition~\ref{prop:dycicleimpliesditriangle}.
So~$G'$ has a ditriangle, which is a separator in~$G$ by Proposition~\ref{prop:trianglesinG-V3areseparator}.
Thus~$G'$ is acyclic, so every dicycle in~$G$ contains a vertex in~$V_3$. 
Let~$C$ be a dicycle in~$G$.

For each~$u \in V_3 \cap V(C)$, let~$e_u$ 
denote the unique arc of~$C$ connecting the two neighbors of~$u$ within~$C$.
Let~$D=(C-V_3) \cup \{e_u:u \in V_3\}.$
As~$V(D) \subseteq V(G')$,~$D$ is not a dicycle.
Hence, there exists~$u\in V_3 \cap V(C)$ such that~$e_u \cdot u$ is a ditriangle.
Thus, by the definition of~$T_2$ and~$T_3$,
~$C \cap T_2 \neq \emptyset$ and 
$C \cap T_3 \neq \emptyset$.
If~$E(C) \cap E(G') \neq \emptyset$, then we are done,
as every arc of~$G'$ is in~$T_1$.
Thus, let us assume that~$E(C) \subseteq E(G \setminus G')$
and suppose by contradiction that 
$C \cap T_1 = \emptyset$.
This implies, by the definition of~$T_2$ and~$T_3$, that 
$C$ consists of an even sequence of arcs that alternates between arcs in~$T_2$ and~$T_3$.
But in this case,~$D$ is a dicycle in~$G'$, a contradiction.
\end{case}
This completes the proof of Theorem~\ref{th:pl3tree}.

\section{Feedback arc sets in~$k$-trees}

Recall that~$\beta(G)$ is the size of a minimum feedback arc set in a digraph $G$.
An easy corollary of Theorem~\ref{th:pl3tree} is that~$\beta(G) \leq m(G)/3$ for any digraph whose underlying graph is a 3-tree~$G$ with~$g(G) \geq 3$.
However, this result is not tight, as any 3-tree~$G$ on~$4$ vertices and~$g(G) \geq 3$ satisfies~$\beta(G) \leq 1$ (Figure \ref{fig:tournaments4vertices}).
In the next theorem, we slightly improve this upper bound and show that it is tight.

\threetreetightTheorem*

\begin{proof}
    By induction on~$n(G)$.
    Let $v_1,v_2,\ldots,v_n$ be a 3-elimination order of $G$. If~$n(G)=4$, then~$G$ has exactly one dicycle and~$\beta(G)=1$ (See Figure \ref{fig:tournaments4vertices}).
    If~$n(G)>4$, by induction hypothesis,~$G'=G-v_n$ has a feedback arc set~$X'$ of size at most~$\frac{m(G')-3}{3}=\frac{m(G)-6}{3}$. 
     Let~$X=X' \cup A$,
     where~$A=\{v_jv_n \in E(G): j<n\}$ if~$|\{v_jv_n \in E(G): j<n\}| \leq 
     |\{v_nv_j \in E(G): j<n\}|$ and~$A=\{v_nv_j \in E(G): j<n\}$ otherwise.
   Clearly,~$X$ is a feedback arc set of~$G$ and~$|A| \leq 1$. Hence
  ~$$
    |X| \leq |X'|+ |A|   \leq 
   \frac{m(G)-6}{3} + 1 = \frac{m(G)-3}{3}.
  ~$$
   For the second part of the theorem, consider a planar 3-tree constructed as follows.
   Start with a ditriangle $G_0$ and a collection of ditriangles~$P_0=\emptyset$.
   Each time, select a ditriangle~$abc$ in~$G_{i-1}$ such that~$|ab \cap E(P)|=2$.
   Add a vertex~$v_i$ and the arcs~$bv_i,v_ia$ and~$v_ic$, creating a new graph~$G_i$.
   Also, set~$P_i=P_{i-1} \cup \{abv_i\}$.
   For this construction to be applicable, it suffices to show 
   that for every~$i\geq 1$, there exists a triangle in~$G_{i-1}$ that intersects the arcs of~$P_{i-1}$ at most twice.
   Proceeding by induction on~$i$, if~$i=1$ then~$P_0=\emptyset$;
       otherwise, by the construction, as~$cv_i \neq E(P_{i})$, we have that 
      ~$|acv_i \cap E(P_i)|=2$.

   Now, note that~$P_n$ has size~$\frac{m(G)-3}{3}$.
   As every pair of triangles in~$P$
   is arc-disjoint, we must have~$\beta(G) \geq |P| \geq \frac{m(G)-3}{3}$.
\end{proof}

Following the first part of the proof of Theorem~\ref{thm:3treetight}, we
can show a similar result for~$k$-trees.
A \emph{tournament} is a digraph in which, for every pair of distinct vertices \(u\) and \(v\), exactly one of the arcs $uv$ or $vu$ is present.
Since every~$k$-tree on~$k+1$ vertices is a tournament, we can use the following result.

\label{prop:betatournament}
\begin{proposition}[{\cite{Poljak1988}}]
    If~$G$ is a tournament
    then~$\beta(G) \leq \frac{m(G)}{2}-c(n(G))^{3/2}$, where~{$c= \frac{1}{8 \sqrt{\pi}}$}.
\end{proposition}

\ktreetightTheorem*

\begin{proof}
By induction on~$n$.
Let $v_1,v_2,\ldots,v_n$ be a $k$-elimination order of $G$.

If~$n = k+1$ then~$G$ is a tournament, so, by Proposition \ref{prop:betatournament}, 
\begin{eqnarray*}
    \beta(G) &\leq& \frac{1}{2}m-cn^{3/2} \\ &=&
    \frac{\floor {k/2}}{k}m + \frac{m}{2k}(\ceil{k/2}-\floor {k/2}) - c(k+1)^{3/2}\\
    &=&
    \frac{\floor {k/2}}{k}m + \frac{k-1}{4}(\ceil{k/2}-\floor {k/2}) - c(k+1)^{3/2}.
\end{eqnarray*}
Now, if $n>k+1$ then we let $G'=G-v_n$. By selecting a set $A$
of at most $\floor {k/2}$ arcs as in the proof of Theorem~\ref{thm:3treetight}, we have
\begin{eqnarray*}
\beta(G) &\leq& \beta(G')+ \floor {k/2} \\
   &\leq&  
   \frac{\floor {k/2}}{k}(m-k) + \frac{k-1}{4}(\ceil{k/2}-\floor {k/2}) - c(k+1)^{3/2}+  \floor {k/2}\\
    &\leq&  
   \frac{\floor {k/2}}{k}m + \frac{k-1}{4}(\ceil{k/2}-\floor {k/2}) - c(k+1)^{3/2}.
\end{eqnarray*}
This completes the proof of the theorem.
\end{proof}


    A trivial upper bound for $\beta(G)$ is $m(G)/2$
(see Remark \ref{remark:dec2acyclic}).
Hence, Theorem \ref{thm:betaktree} becomes more interesting for odd values of $k$.

\begin{corollary}
If $G$ is a $k$-tree with $k$ odd, then $\beta(G) \leq \frac{k-1}{2k}m(G) + \frac{k-1}{4} - c(k+1)^{3/2}$, where~$c= \frac{1}{8 \sqrt{\pi}}$.
\end{corollary}

We observe that, in contrast to Theorem~\ref{thm:3treetight}, we have not shown that the result of Theorem~\ref{thm:betaktree} is asymptotically tight.
We believe that this is not the case.
Indeed, as~$g(G) \geq 3$,
any collection of pairwise arc-disjoint directed cycles has size at most~$m(G)/3$,
so, using a technique similar to the one used in Theorem~\ref{thm:3treetight}, we can only hope to obtain a lower bound of~$m(G)/3$ for $\beta(G)$.

\section{Final remarks and open problems}

In this paper, we showed Woodall's conjecture (Conjecture~\ref{conj:Woodall-planar}) for digraphs whose underlying graph is a 3-tree. This result is a step towards proving Woodall’s conjecture for digraphs whose underlying graph is a partial $k$-tree. In particular, the conjecture for partial 3-trees remains open.
In this direction, we show that if the girth of a partial 3-tree is~3, then the conjecture holds.

Woodall's conjecture is difficult to tackle in its current form, as, if the conjecture is true, then a direct corollary is an upper bound for the size of a minimum feedback arc set. And, to our knowledge, nontrivial upper bounds for this parameter are only known for digraphs with small maximum degree.
In this direction, we progress by proving a nontrivial upper bound for~$k$-trees, and a tight upper bound for 3-trees. However, we do not know if the bound for $k$-trees can be improved.

\begin{problem}
    Improve the result of Theorem~\ref{thm:betaktree} or show that this result is also tight.
\end{problem}

As we mentioned earlier, we are not aware of any nontrivial upper bound for partial $k$-trees.
In this direction, Knauer et al.~\cite{Knauer2022} study upper bounds for the size of a minimum feedback vertex set for~$k$-degenerate graphs, which are a generalization of partial~$k$-trees.
We think some of the techniques presented here can be extended to study upper bounds for minimum feedback arc sets in~$k$-degenerate graphs.

\begin{problem}
    Find a nontrivial upper bound for~$\beta(G)$ when~$G$ is~$k$-degenerate, chordal or a partial~$k$-tree.
\end{problem}

The tightness of Theorem~\ref{thm:3treetight} relies implicitly on the concept of \emph{packing of dicycles}, which is a pairwise set of arc-disjoint dicycles. We use the fact that the size of any maximum packing of dicycles in a digraph~$G$, denoted by~$\nu_C(G)$, is at most $\beta(G)$.
The dual version of the celebrated Theorem of Luchessi and Younger \cite{Lucchesi1978} states that equality applies if $G$ is planar.The counterexample of Thomassen mentioned in the introduction showd that equality does not hold for nonplanar digraphs. Indeed, in this case,
$\beta(G) > m(G)/3$, but $nu_C(G) \leq m(G)/3$.
However,
as in the dual version of Woodall’s conjecture, we believe it is interesting to study a nonplanar version of the dual of the Lucchesi–Younger theorem.

\begin{problem}
Show that~$\nu(G)=\beta(G)$ for every~$k$-degenerate graph for small values of $k$.
\end{problem}

Moreover, we can relax the conjecture as follows.

\begin{problem}
Show that there exists a constant $c$ such that~$\beta(G) \leq c \cdot \nu_C(G)$ for every digraph~$G$.
\end{problem}

Returning to Woodall's conjecture, we can define, for every~$k \in \mathbb{N}$ with~$k\geq 2$, the following families of conjectures
\footnote{see also \url{http://www.openproblemgarden.org/op/woodalls_conjecture}}.

\begin{conjecture}[Conjecture~$WC_{=}\mathbf{(k)}$]\label{conj:WC=}
Let~$G$ be a planar digraph. If~$g(G)=k$ then~$\nu(G)=k$.
\end{conjecture}

\begin{conjecture}[Conjecture~$WC_{\geq}\mathbf{(k)}$]\label{conj:WC>=}
Let~$G$ be a planar digraph. If~$g(G)\geq k$ then~$\nu(G) \geq k$.
\end{conjecture}

Note that if we solve~$WC_{=}\mathbf{(k)}$ for any~$k$, then
Woodall's conjecture holds.
Also, note that a solution for~$WC_{\geq}\mathbf{(k)}$ 
implies a solution for~$WC_{=}\mathbf{(k)}$.
In fact, if that is the case, then~$k \leq \nu(G) \leq g(G)=k$.
Conjecture~$WC_{\geq}\mathbf{(2)}$ is equivalent to the problem
of decomposing the arcs of a digraph into two acyclic digraphs.
There is a folklore solution to this problem: consider an arbitrary sequence of the nodes and divide the set of arcs of the digraph into two sets, depending on whether it consists of an increasing or decreasing pair according to the sequence.

\begin{remark}\label{remark:dec2acyclic}
The set of arcs of any digraph~$G$ can be decomposed into two
acyclic digraphs.
\end{remark}

Extending this result, Wood proved that any digraph can be decomposed into an arbitrary number of acyclic digraphs. Moreover, any vertex in such digraph has a small outdegree proportional to its original outdegree \cite{Wood2004}.
However, the families of conjectures given here are more difficult to prove, as they ask us to decompose
the arcs of a digraph into~$k$ acyclic digraphs, but with the additional condition that the union of any~$k-1$ of these digraphs is also acyclic.

No result is currently known for Conjecture~$WC_{\geq}\mathbf{(3)}$.
As we mention before, Conjecture~$WC_{\geq}\mathbf{(2)}$,
and thus Conjecture~$WC_{=}\mathbf{(2)}$, are already settled,
and we answered Conjecture~$WC_{=}\mathbf{(3)}$ for partial 3-trees in Corollary~\ref{cor:betam3partial3trees}.
No more results are known on this line.
We believe that this is evidence that Woodall’s conjecture is difficult in its current form and that a project to prove the full conjecture should be structured around small incremental steps.

\begin{problem}
Study Conjectures~\ref{conj:WC=} and~\ref{conj:WC>=} for particular graph classes and for particular values of~$k$.
\end{problem}


\begin{thebibliography}{99}

\bibitem{Woodall78}
D. R. Woodall.
\newblock Menger and König systems.
\newblock In \emph{Theory and Applications of Graphs}, pages 620--635, 1978.



\bibitem{Lee2001}
O. Lee and Y. Wakabayashi.
\newblock Note on a min-max conjecture of Woodall.
\newblock \emph{Journal of Graph Theory}, 38(1):36--41, 2001.



\bibitem{Lee2006}
O. Lee and A. Williams.
\newblock Packing dicycle covers in planar graphs with no $K_5-e$ minor.
\newblock In \emph{Lecture Notes in Computer Science}, volume 3887, pages 677--688, 2006.



\bibitem{Arnborg1990}
S. Arnborg, A. Proskurowski, and D. G. Corneil.
\newblock Forbidden minor characterization of partial 3-trees.
\newblock \emph{Discrete Mathematics}, 80(1):1--19, 1990.



\bibitem{Feofiloff87}
P. Feofiloff and D. Younger.
\newblock Directed cut transversal packing for source-sink connected graphs.
\newblock \emph{Combinatorica}, 7:255--263, 1987.


\bibitem{Robertson90}
N. Robertson and P. D. Seymour.
\newblock Graph minors. {IV}. {T}ree-width and well-quasi-ordering.
\newblock {\em Journal of Combinatorial Theory, Series B}, 48(2):227--254, 1990.

\bibitem{Bodlaender88}
H. L. Bodlaender.
\newblock Dynamic programming on graphs with bounded treewidth.
\newblock In \emph{Automata, Languages and Programming}, volume 310, pages 105--118, 1988.



\bibitem{Arnborg89}
S. Arnborg and A. Proskurowski.
\newblock Linear time algorithms for NP-hard problems restricted to partial $k$-trees.
\newblock \emph{Discrete Applied Mathematics}, 23(1):11--24, 1989.


\bibitem{Brandstadt99}
A. Brandstädt, V. B. Le, and J. P. Spinrad.
\newblock \emph{Graph Classes: A Survey}.
\newblock SIAM, 1999.


\bibitem{Bodlaender98}
H. L. Bodlaender.
\newblock A partial $k$-arboretum of graphs with bounded treewidth.
\newblock \emph{Theoretical Computer Science}, 209(1):1--45, 1998.



\bibitem{Wood2004}
D. R. Wood.
\newblock Bounded degree acyclic decompositions of digraphs.
\newblock {\em Journal of Combinatorial Theory, Series B}, 90(2):309--313, 2004.


\bibitem{Schrijver1982}
A. Schrijver.
\newblock Min-max relations for directed graphs.
\newblock In \emph{Bonn Workshop on Combinatorial Optimization}, volume 66, pages 261--280, 1982.


\bibitem{Cornuejols2025}
G. Cornuéjols, D. Liu, and R. Ravi.
\newblock Packing dijoins in chordal digraphs.
\newblock \emph{arXiv preprint} arXiv:2501.10918, 2025.
\newblock \url{https://arxiv.org/abs/2501.10918}.

\bibitem{Donadelli2002}
J. Donadelli  and Y. Kohayakawa.
\newblock A density result for random sparse oriented graphs and its relation to a conjecture of Woodall.
\newblock \emph{Electron. J. Combin.}, 9(1):R45, 2002.
\newblock \url{https://doi.org/10.37236/1661}.

\bibitem{Gutin2025}
G. Gutin, M. A. Nielsen, A. Yeo, and Y. Zhou.
\newblock Feedback arc sets and feedback arc set decompositions in weighted and unweighted oriented graphs.
\newblock \emph{arXiv preprint} arXiv:2501.06935, 2025.
\newblock \url{https://arxiv.org/abs/2501.06935}.


\bibitem{Poljak1988}
S. Poljak, V. Rödl, and J. Spencer.
\newblock Tournament ranking with expected profit in polynomial time.
\newblock \emph{SIAM J. Discrete Math.}, 1(3):113--122, 1988.


\bibitem{Knauer2022}
K.~Knauer, T.-M.~La, and A.~Valicov,
\newblock Feedback vertex sets in (directed) graphs of bounded degeneracy or treewidth,
\newblock \emph{Electron. J. Combin.}, 29(4):P4.16, 2022.

\bibitem{Lucchesi1978}
C.~L.~Lucchesi and D.~H.~Younger,
\newblock A minimax theorem for directed graphs,
\newblock \emph{J. London Math. Soc. (2)}, 17(3):369--374, 1978.


\end{thebibliography}
\end{document}